\documentclass[12pt, reqno]{amsart}
 \usepackage{amsmath, amsfonts, amssymb, amscd, mathrsfs,color, graphicx, amsthm}
\newtheorem{theorem}{Theorem}[section]
\newtheorem{proposition}{Proposition}[section]

\newtheorem{lemma}{Lemma}[section]
\newtheorem{remark}{Remark}[section]

\makeatletter \@addtoreset{equation}{section}

\begin{document}
\title[Berezin transforms on line bundles over $\mathbb{B}^{n}$ ]{On the Berezin transforms on line bundles over the complex Hyperbolic spaces}
\author[N. Askour] {Nour eddine Askour}
\address{ Department of Mathematics, Sultan My Slimane University, Faculty
of Sciences and Technics (M'Ghila), Beni Mellal, Morocco.}
\email{\textcolor[rgb]{0.00,0.00,0.84}{n.askour@usms.ma}}
.

\begin{abstract}
we define a generalized Berezin transforms on line bundle over the complex hyperbolic space $\mathbb{B}^{n}=S(n,1)/S(U(n)\times U(n),$ and we give it as a functions of the G-invariant laplacian on the line bundles.
\end{abstract} \maketitle

\section{Introduction}
 Berezin transform \cite{bere:74} is of particular interest both in quantum theory and operator theory\cite{MREN:99},\cite{Vasi:08}. In a series of papers \cite{bere:75}, \cite{Karab:12} and references theirs in, F. Berezin introduce a new approach to quantization of Khaler manifolds, based on reproducing kernel function theory. Since then many authors have been interested in the so-called Berezin quantization.\\

  Classically the Berezin transform is defined as follows. Consider a domain $\Omega $ $\subset \Bbb{C}^{n}$ and a Borel measure $d\mu $ on $\Omega $. Let $\mathfrak{H}$ be a closed subspace of $L^{2}\left( \Omega ,d\mu \right) $ consisting of continuous function and assume that $\mathfrak{H}$ has a reproducing kernel $K\left( .,.\right) $. The Berezin symbol $\widehat{A}$ of a bounded operator $A$ on $\mathfrak{H}$ is the function defined on $\Omega .$ by

\begin{equation}
\hat{A}\left( z\right) =\frac{\left\langle AK\left( .,z\right) ,K\left( .,z\right) \right\rangle }{K\left( z,z\right) },\qquad z\in \Omega .
\end{equation}
For each $\varphi$ such that $\varphi\mathfrak{H}\in L^{2}\left( \Omega ,d\mu \right)$-for instance for any $\varphi\in L^\infty(\Omega)$, the Toeplitz operator $T_{\varphi}$ with symbol $ \varphi$ is the operator on $\mathfrak{H}$ given by $T_\varphi[f]=P(f\varphi)$; $f\in\mathfrak{H}$, where P is the orthogonal projector on $\mathfrak{H}$.
By definition the Berezin transform B is the integral transform defined by
\begin{equation}
 B[\varphi](z):=\widehat{T_{\varphi}}(z)=\int_{\Omega}\frac{|K(z,\omega)|^{2}}{K(z,z)}\varphi(\omega)d\mu(\omega).
 \end{equation}

The formula representing the Berezin transform as a function of Laplace Beltrami operator plays an important role in the Berezin quantization theory \cite{MREN:96}.

In the case of a bounded symmetric domain $\mathfrak{D}=G/K$, the Berezin transform intertwines with the group actions.Therefore it is a function, in the sense of the functional calculus for commuting self-adjoint operators, of the $G$-invariant differential operators $ \Delta_{1},...,\Delta_{r}$ generating the algebra of all $G$-invariant differential operators on  $\mathfrak{D}$. This idea was carried out by Berezin \cite{bere:75} in the rank one case (without proof) and proved by Unterberger and Upmeier \cite{unup:94} for the general case in the strongest, spectral-theoretic, sense \cite{UPME:97}.\\

Now, taking into a count that the Berezin transform can be defined provided that there is a given closed subspace, which possesses a reproducing kernel, we are here concerned with the domain $\Omega=\Bbb B^{n}$ the unit ball of $\Bbb C^{n}$ viewed  as homogeneous space $ \Bbb B^{n}= G/K$ where $G=SU(n,1)$ be the group of $\Bbb C-$ linear transforms $g$ on $\Bbb C^{n+1}$ that preserve the the indefinite Hermitian form $\sum_{j=1}^{n}|z_{j}|^{2}-|z_{n+1}|^{2},$ with $detg=1.$ Suppose that $\Bbb B^{n}$ is endowed with the measure $d\mu_{\nu}=(1-|z|^{2})^{\nu-n-1}dm(z),$ where $\nu\in \Bbb R_{+} \backslash \Bbb Z,$ $\nu>n$ and $dm(z),$ be the Lebesgue measure on $\mathbb{C}^{n}.$ We consider as the the subspace of the space of the $L^{2}-$ integrable functions on  $\mathbb{B}^{n}$ with respect to the measure $d\mu_{\nu},$ the range space of the spectral projector $R^{\nu}_{l}$ corresponding the eigenvalue $\rho_{l}=-(\nu-n-2l)^{2}$ of the $G$-invariant differential operator:

 \begin{equation}
 \Delta_{\nu}=4(1-|z|^{2})\{\sum_{1\leq i,j\leq n}(\delta_{i,j}-z_{i}\overline{z}_{j})\frac{\partial^{2}}{\partial z_{i}\partial\overline{z}_{j}}-\nu \sum_{j=1}^{n}\overline{z}_{j}\frac{\partial}{\partial\overline{z}_{j}}\},
 \end{equation}
 where $l$ is a fixed integer such that $0\leq l<\frac{\nu-n}{2}.$ In other word, we are concerned with the following eigenspace:

 \begin{equation}
A^{2,\nu}_{l}=\{F\in L^{2}(\Bbb B^{n},d\mu_{\nu}),\Delta_{\nu}F=\rho_{l}F\}.
 \end{equation}

 \begin{equation}
K_{l}^{\nu}(z,w)=c_{l}\frac{1}{(1-<z,w>)^{\nu}}\quad_{2}F_{1}(-l,l-\nu+n,n;1-\frac{|1-<z,w>|^{2}}{(1-|z|^{2})(1-|w|^{2})}).
\end{equation}
where,
\begin{equation}
c_{l}=\frac{2\Gamma(n+l)}{\pi^{n}\Gamma(n)l!}\frac{(\nu-n-2l)\Gamma(\nu-l)}{\Gamma(\nu-n-l+1)}.
\end{equation}

It is known, that the kernel $K_{0}(z,w)$, corresponding to $l=0,$ is the Bergman kernel, hence the space $A^{2,\nu}_{0}$ is the classical weighted Bergman- space of holomorphic functions that are $(1-|z|^{2})^{\nu-n-1}dm(z)$-integrable, while for $l\neq0$, The space $A^{2,\nu}_{l}$ which can be viewed as as a Kernel spaces of the elliptic differential operator $\Delta_{\nu}-\rho_{l}$, consists of non holomorphic functions. The Berezin transform associated the Hilbert subspace $A^{2,\nu}_{0}$ is studied by many Authors \cite{Peet:90}.\\ Here, for the case $l\neq0,$  we associate to the sub space $A^{2,\nu}_{l},$ the following Berezin transform :

$$L^{2}_{\nu}(\Bbb B^{n})\rightarrow L^{2}_{\nu}(\Bbb B^{n})$$

\begin{equation}
B^{l}_{\nu}F(z)=\int_{\mathbb{B}^{n}} F(w)B^{\nu}_{l}(z,w) d\mu_{\nu}(w),
\end{equation}

where the kernel $B^{\nu}_{l}(z,w),$ is defined by:
\begin{equation}
B^{\nu}_{l}(z,w)=(1-<z,w>)^{-\nu}\frac{\mid K^{\nu}_{l}(z,w)\mid ^{2}}{K^{\nu}_{l}(z,z)K^{\nu}_{l}(w,w)}, (z,w)\in\mathbb{B}^{n}\times \mathbb{B}^{n}.
\end{equation}

In this paper, our aim is to express the above Berezin transform as a function of the $ G-$ invariant Laplacian $\Delta_{\nu}.$ The method used is based on the $L^{2}-$ spectral theory of $\Delta_{\nu}.$ \cite{boin:98}, \cite{zang:92}, together with the Fourier-Jacobi transform \cite{koor:84},\cite{kaji:01}. Precisely, we establish the following result:

$$
B^{\nu}_{l}=\frac{\pi^{n}\Gamma^{2}(n)}{2\Gamma^{2}(\nu-l)}|\Gamma(\frac{i\sqrt{-(\Delta_{\nu}+(n-\nu)^{2})}+3\nu-4l-n}{2})|^{2}
$$

\begin{equation}
\times\sum_{q=0}^{2l}(-1)^{q}A_{q}\frac{S_{q}(\frac{-(\Delta_{\nu}+(n-\nu)^{2})}{4},\frac{3\nu-4l-n}{2},\frac{\nu+n}{2},\frac{n-\nu}{2})}{\Gamma(2(\nu-l)+q)\Gamma(\nu-2l+q)},
\end{equation}

where $S_{q}$ denotes the continuous dual Hahn polynomial and $A_{q}$ are the following parameters
\begin{equation}
A_{q}=2^{-q}\sum_{p=\max (0,q-l)}^{\min (l,q)}(_{q-p}^{l})(_{p}^{l})\frac{\Gamma(\nu-l+q-p)\Gamma(\nu-l+p)}{\Gamma(n+p)\Gamma(n+q-p)}.
\end{equation}
The paper is organized as follows. In section 2, we review some well known spectral properties of the operator $\Delta_{\nu}$. In the section 3 we give the spectral density associated with the G-invariant Laplacian $\Delta_{\nu}$. As an application we give its heat kernel. The section 4, will be devoted for the $G-$ invariance and the boundedness of the Berezin transform $B^{\nu}_{l}.$  In the section 5, we give the proof of the main result (1.9). In section 6, as an application we give an expression of the Berezin heat kernel.

\section{ $L^{2}$-Concrete spectral analysis of the invariant Laplacians $\Delta_{\nu}$.}

In this section we review some results on the $L^{2}$-Concrete spectral analysis,in the sense of Strichartz \cite{stro:89}, of the invariant Laplacians $\Delta_{\nu}$ in the weighted Hilbert space $L^{2}_{\nu}(\Bbb B^{n}).$\\

Let $G=SU(n,1)$ be the group of all $\mathbb{C}$-linear transforms $g,$ on $\mathbb{C}^{n+1}$ that preserve the indefinite hermitian form
\begin{equation}
\sum^{n}_{j=1}  \mid z_{j}\mid^{2}-\mid z_{n+1}\mid^{2},
\end{equation}
with $\det g=1$. \\
The group $G$ acts transitively on the unit ball $\mathbb{B}^{n}=\{z\in \mathbb{C}^{n}; \mid z\mid<1\}$ by
\begin{equation}
G\ni g=\left(
\begin{matrix}
a & b \\
c & d
\end{matrix}
\right)
:z\rightarrow g.z=(az+b)(cz+d)^{-1},
\end{equation}
where $a,b,c,d$ are $n\times n$, $n\times1$, $1\times n$ and $1\times1$ matrices respectively.
Recall that this action satisfy the following relation:
\begin{equation}
1-<gz,gw>=\frac{(1-<z,w>)}{(cz+d)(\overline{cw+d})}
\end{equation}
where $<,>$ is the well known hermitian product on $\mathbb{C}^{n}.$\\
As a homogeneous space we have the identification $\mathbb{B}^{n}=G/K$ where $K$ is the stabilizer of $0$. More precisely
\begin{equation}
K=\left\{ k=\left(
\begin{matrix}
a & 0 \\
0 & d
\end{matrix}
\right)
,a\in U(n), d\in U(1) \quad ; \det(ad)=1\right\}.
\end{equation}
We recall that The $G-$ invariant distance associated to the Bergman metric \cite{chab:90} on the unit ball $\mathbb{B}^{n}=G/K,$ is given by:

\begin{equation}\label{eq:hyperbolicdistance}
\cosh^{2} d(z,w)=\frac{\mid 1-<z,w>\mid^{2}}{(1-\mid z\mid^{2})(1-\mid w\mid^{2})}, (z,w)\in\mathbb{B}^{n}\times\mathbb{B}^{n}.
\end{equation}

Let $\nu \in \mathbb{R}-\mathbb{Z},$ and suppose that $\nu>0$. By $dm(z)$ we denote the Lebesgue measure on $\mathbb{C}^{n}.$
Denote by $d\mu_{\nu},$ the weighted measure on $ \mathbb{B}^{n}$ defined by:

\begin{equation}
d\mu_{\nu}(z)=(1-\mid z\mid^{2})^{\nu-n-1}dm(z),
\end{equation}

and by $L_{\nu}^{2}(\mathbb{B}^{n})$ its the corresponding $L^{2}$-space,

\begin{equation}
L_{\nu}^{2}(\mathbb{B}^{n})=\left\{F:\mathbb{B}^{n}\mapsto \mathbb{C},\int_{\mathbb{B}^{n}}|F(z)|^{2}d\mu_{\nu}(z)<+\infty\right\}
\end{equation}

 For $g\in G,$ we define

\begin{eqnarray}\label{eq:gaction}
T^{\nu}(g)F(z)=J(g^{-1},z)^{\frac{\nu}{n+1}}   F(g^{-1}.z),
\end{eqnarray}

where $J(g^{-1},z)$ is the complex Jacobian of $g^{-1}$ (with a mild ambiguity of its $\nu$ power depending only on $g$).\\
Then $T^{\nu}$ gives rise to a continuous projective  representation of the group $G$  on $L^{2}_{\nu}(\mathbb{B}^{n})$.\\
Notice that the restriction of $J$ to $K$ gives rise to a character $\chi_{\nu}$ of $K$.\\ Namely,

\begin{eqnarray}
J(k,z)^{\frac{\nu}{n+1}}=d^{-\nu},
\end{eqnarray}

for $ k=\left(\begin{matrix}a & 0 \\
0 & d
\end{matrix}
\right)$
.\\
The space $L^{2}_{\nu}(\mathbb{B}^{n})$ is a trivialization of the $L^{2}$-space of sections of the homogeneous line bundle
over $\mathbb{B}^{n}$ associated to the one dimensional representation  $\chi_{\nu}$ of the compact group $K$.\\

The invariant Laplacian with respect to the $G$-action  (\ref{eq:gaction}) is given by:
\begin{eqnarray}
\Delta_{\nu}=4(1-\mid z\mid^{2})\{\sum_{\leq1i,j\leq n}(\delta_{ij}-z_{i}\bar{z}_{j})\frac{\partial^{2}}{\partial z_{i}\partial \bar{z}_{j}}-\nu\sum_{j=1}^{n}\bar{z}_{j}\frac{\partial}{\partial \bar{z}_{j} }\}.
\end{eqnarray}
\textbf{Remark.}
Note that, in~\cite{boin:98}, a more general family of Laplacians $\Delta_{\alpha,\beta}$ has been considered.
The above operators $\Delta_{\nu}$ corresponds to the case \\$\alpha=0$ and  $\beta=-\nu$.\\

In\cite{boin:98} we showed that the invariant laplacian $\Delta_{\nu}$ is a self-adjoint operator in the space $L^{2}_{\nu}(\mathbb{B}^{n}).$\\
Besides the continuous spectrum $\{-(\lambda^{2}+(\nu-n)^{2}),\lambda\in\mathbb{R}\},$ it might have a discrete spectrum according to the size of $\nu$.\\
 Precisely, if $\nu>n$ then the point spectrum of the $G$-invariant laplacian $\Delta_{\nu}$  consists of the finite set
\begin{equation}
\rho_{j}=-(\lambda_{j}^{2}+(\nu-n)^{2}),j=0,.....,[\frac{\nu-n}{2}]
\end{equation}

where
\begin{equation}\label{eq:lambdadeboussejra}
\lambda_{j}=i(2j+n-\nu),
\end{equation}

and, $[x]$=the greatest integer not exceeding $x.$\\
Thus, in the case where $\nu>n,$ the spectrum $\sigma(\Delta_{\nu})$ of the the operator $\Delta_{\nu}$ is given by

\begin{equation}\label{eq:spectrumofnu}
 \sigma(\Delta_{\nu})=\{-(\lambda^{2}+(\nu-n)^{2}),\lambda\in\mathbb{R}\}\cup\{-(\lambda_{j}^{2}+(\nu-n)^{2});j=0,.....,[\frac{\nu-n}{2}]\}
 \end{equation}

According to  \cite{boin:98} and \cite{zang:92} a fundamental family of eigenfunctions of $\Delta_{\nu}$ with eigenvalue $-(\lambda^2+(\nu-n)^2)$
is given by the following family of Poisson kernels:

\begin{equation}\label{eq:poissonkernel}
P^{\nu}_{\lambda}(z,\omega)=(\frac{1-\mid z\mid^{2}}{\mid 1-<z,\omega>\mid^{2}})^\frac{i\lambda+n-\nu}{2}(1-<z,\omega>)^{-\nu},
\end{equation}

from which we may obtain an explicit spectral decomposition of the self-adjoint operator $\Delta_{\nu}$ in the Hilbert space $L^{2}_{\nu}(\mathbb{B}^{n})$.\\
More precisely, let $F\in L^{2}_{\nu}(\mathbb{B}^{n})$. Then we have

\begin{equation}\label{eq:spectraldecomposition}
F=\int^{+\infty}_{-\infty} \mathcal{P}_{\lambda}^{\nu}Fd\lambda +\sum_{0\leq j<\frac{ \nu-n}{2}}\mathcal{R}^{\nu}_{j}F,
\end{equation}

where the integral operators $\mathcal{P}_{\lambda}^{\nu}$ are related to the Fourier-Helgason transform
\begin{equation}
\tilde{F}(\lambda,\omega)=\int_{\mathbb{B}^{n}}F(z){P^{\nu}_{-\lambda}(z,\omega)}d\mu_{\nu}(z),
\end{equation}
by
\begin{equation}\label{eq:projectorlambda}
\mathcal{P}_{\lambda}^{\nu}F(z)=\frac{\Gamma(n)}{42^{2(\nu-n)}\pi^{n+1}}  \mid c^{\nu}(\lambda)\mid^{-2}\int_{\partial
\mathbb{B}^{n}}P^{\nu}_{\lambda}(z,\omega)\tilde{F}(\lambda,\omega)d\sigma(\omega).
\end{equation}

where
\begin{eqnarray}\label{eq:harichanda}
c_{\nu}(\lambda)=\dfrac{2^{-\nu+n-i\lambda}\Gamma(n)\Gamma(i\lambda)}{\Gamma(\frac{i\lambda+n-\nu}{2})\Gamma(\frac{i\lambda+n+\nu}{2})},
\end{eqnarray}

is the analogous of the Harish-Chandra c-function.\\
In above the orthogonal projector operators $\mathcal{R}^{\nu}_{j}$ are given by

\begin{equation}\label{eq:projectorl}
\mathcal{R}^{\nu}_{j}F(z)=c_{j}\int_{\partial \mathbb{B}^{n}}\tilde{F}(\lambda_{j},\omega)P^{\nu}_{\lambda_{j}}(z,\omega)d\omega,
\end{equation}

where
\begin{eqnarray}\label{eq:constantreproduce}
c_{j}=\frac{2\Gamma(n+j)}{\pi^{n}\Gamma(n)j!}\frac{(\nu-n-2j)\Gamma(\nu-j)}{\Gamma(\nu-n-j+1)}.
\end{eqnarray}

Note that the poisson kernel defined in (\ref{eq:poissonkernel}) satisfy the following integral formula:

$$
\int_{\partial\mathbb{B}^{n}}P^{\nu}_{\lambda}(z,\omega)P^{\nu}_{-\lambda}(w,\omega)d\sigma(\omega)=
$$

\begin{equation}\label{eq:integralpoisson}
(1-<z,w>)^{-\nu}F(\frac{i\lambda+n-\nu}{2},\frac{i\lambda+n-\nu}{2},n,-\sinh^{2} d(z,w))
\end{equation}

\begin{remark}
Notice that the set $\{\lambda_{j}=i(2j+n- \nu),0\leq j<\frac{\nu-n}{2}\}$ corresponds to the poles of the Harish-Chandra c-function $c_{\nu}(\lambda)^{-1}$ in the region $Im\lambda<0$.
\end{remark}

From now on we suppose that $\nu>n$.
\section{Spectral density}
Recall that our aim in this paper is to express the Berezin transform $B_{\nu}^{l}$ as function of the $G-$ invariant Laplacian $\Delta_{\nu}.$ For this, we consider the $G-$ invariant shifted Laplacian defined by:
\begin{eqnarray}
 \widetilde{\Delta}_{\nu}=-(\Delta_{\nu}+(\nu-n)^{2}),
 \end{eqnarray}
with $\mathcal{C}_{0}^{\infty}(\mathbb{B}_{n}),$ as its natural regular domain.
Note that the spectrum of the operator $\widetilde{\Delta}_{\nu}$ can be given easily from (\ref{eq:spectrumofnu}) by:

 \begin{equation}\label{eq:spectrumtild}
 \sigma(\widetilde{\Delta_{\nu}})=\{s=\lambda^{2},\lambda\in\mathbb{R}\}\cup\{s_{j}=\lambda_{j}^{2};j=0,.....,[\frac{\nu-n}{2}]\};
 \end{equation}
where $\lambda_{j}$ are defined in (\ref{eq:lambdadeboussejra}). To express The Berezin transform $B_{\nu}^{l}$ in term of $\widetilde{\Delta}_{\nu},$ we will need to compute its spectral density. The extension of $\widetilde{\Delta}_{\nu}$ will be also denoted by $\widetilde{\Delta}_{\nu}.$ The domain of the extension $\widetilde{\Delta}_{\nu}$ will be denote by $\chi.$ This extension admits a spectral decomposition \cite{yosi:68}:
\begin{equation}\label{eq:resoluidentity}
I=\int_{-\infty}^{+\infty}dE_{s},
\end{equation}
where $I$ is the identity operator and
\begin{equation}\label{eq:spectraldecomposition}
\widetilde{\Delta}_{\nu}=\int_{-\infty}^{+\infty}sdE_{s},
\end{equation}
in the weak sense, that is
\begin{equation}
(\widetilde{\Delta}_{\nu}f,g)=\int_{-\infty}^{+\infty}sd(E_{s}f,g)
\end{equation}
for $f\in \chi$ and $g\in \mathrm{L}_{\nu}^{2}(\mathbb{B}_{n}).$ The spectral density \cite{esfu:99}:
\begin{equation}
e_{s}=\frac{dE_{s}}{ds}
\end{equation}
is understood as an operator-valued distribution, an element of the space $\mathcal{D'}(\mathbb{R}, L(\chi,\mathrm{L}_{\nu}^{2}(\mathbb{B}_{n}))$
where $L(\chi,\mathrm{L}_{\nu}^{2}(\mathbb{B}_{n})),$ is the space of bounded operators from $\chi$ to $\mathrm{L}_{\nu}^{2}(\mathbb{B}_{n}).$
In term of the spectral density $e_{s}=\frac{dE_{s}}{ds},$ the equation (\ref{eq:resoluidentity}) and (\ref{eq:spectraldecomposition}) become
\begin{equation}
I=<e_{s},1>,
\end{equation}
and
\begin{equation}
\widetilde{\Delta}_{\nu}=<e_{s},s>,
\end{equation}
where $<f(s),\phi(s)>$ is the evaluation of the distribution $<f(s),\phi(s)>,$ is the evaluation of the distribution $f(s)$ on a test function $\phi(s).$ Since $\widetilde{\Delta}_{\nu}$ is elliptic, then its spectral density $e_{s}$ admits a distributional kernel (called the spectral function) $e(s,z,w),$ an element of $\mathcal{D}'(\mathbb{R},\mathcal{D}'(\mathbb{B}_{n}\times \mathbb{B}_{n})).$ Precisely, we have the following proposition.





\begin{proposition}\label{eq:spectralfunction}
 The spectral function $e(s,w,z)$ of the operator $\widetilde{\Delta}_{\nu},$ is given by:

$$e(s,w,z)=\frac{\Gamma(n)}{4\pi^{n+1}2^{2(\nu-n)}}(1-<z,w>)^{-\nu}\chi_{+}(s)|C_{\nu}(\sqrt{s})|^{-2}(\sqrt{s})^{-1}\phi_{\sqrt{s}}^{(n-1,-\nu)}(d(z,w))$$

$$
+\sum_{j=0}^{\frac{\nu-n}{2}}c_{j}(1-<z,w>)^{-\nu}\phi_{\sqrt{s_{j}}}^{(n-1,-\nu)}(d(z,w)\delta(s-s_{j}),
$$

where $\chi_{+}(s)$ is the characteristic function of the set of real positif numbers,$s$ and $s_{j}$ are the spectral parameters defined in (\ref{eq:spectrumtild}) and $\phi^{(\alpha,\beta)}_{\lambda}(t),$ is the Jacobi function defined by:

$$
\phi^{(\alpha,\beta)}_{\lambda}(t)=\quad_{2}F_{1}(\frac{\alpha+\beta+1-i\lambda}{2},\frac{\alpha+\beta+1+i\lambda}{2},1+\alpha;-\sinh^{2}t).\\
$$

\end{proposition}
$\mathbf{Proof}$

Let $F$ be a  $C^{\infty}-$ function with compact support in $\mathbb{B}^{n},$ then from the relation (\ref{eq:spectraldecomposition}), we have:

\begin{equation}\label{eq:comp1}
F(z)=\int^{+\infty}_{-\infty} \mathcal{P}_{\lambda}^{\nu}[F](z)d\lambda +\sum_{0\leq j<\frac{ \nu-n}{2}}\mathcal{R}^{\nu}_{j}[F](z).
\end{equation}

Now, by inserting  (\ref{eq:projectorlambda}) and (\ref{eq:projectorl}) in the equation (\ref{eq:comp1}), we obtain:

$$
F(z)=\frac{\Gamma(n)}{4\pi^{n+1}2^{2(\nu-n)}}\int_{-\infty}^{+\infty}d\lambda |C_{\nu}(\lambda)|^{-2}\int_{\mathbb{B}^{n}}\left(\int_{\partial\mathbb{B}^{n}}P^{\nu}_{\lambda}(z,\omega)P^{\nu}_{-\lambda}(w,\omega)d\sigma(\omega)\right)F(w)d\mu_{\nu}(w)
$$

\begin{equation}\label{eq:comp2}
+\sum_{0}^{\frac{\nu-n}{2}}c_{j}\int_{\mathbb{B}^{n}}\left(\int_{\partial\mathbb{B}^{n}}P^{\nu}_{\lambda_{j}}(z,\omega)P^{\nu}_{-\lambda_{j}}(w,\omega)d\sigma(\omega)\right)F(w)d\mu_{\nu}(w).
\end{equation}

where $d\sigma(\omega),$  is the superficial measure on $\partial\mathbb{B}^{n}.$\\
Making use of the formula(\ref{eq:integralpoisson})in where the hypergeometric function in the right hand side, was replaced by the corresponding Jacobi function to get:

$$
F(z)=\frac{\Gamma(n)}{4\pi^{n+1}2^{2(\nu-n)}}\int_{-\infty}^{+\infty}d\lambda |C_{\nu}(\lambda)|^{-2}\int_{\mathbb{B}^{n}}(1-<z,w>)^{-\nu}(\phi^{(n-1,-\nu)}_{\lambda}(d(z,w)))F(w)d\mu_{\nu}(w)
$$

\begin{equation}\label{eq:comp3}
+\sum_{0}^{\frac{\nu-n}{2}}c_{j}\int_{\mathbb{B}^{n}}(1-<z,w>)^{-\nu}(\phi^{(n-1,-\nu)}_{\lambda_{j}}(d(z,w)))F(w)d\mu_{\nu}(w).
\end{equation}

It is not difficult to see that the function involved in the first integral with respect to the variable $\lambda$ in (\ref{eq:comp3}) is even. Then the equation (\ref{eq:comp3}) can be written as:
$$
F(z)=\frac{\Gamma(n)}{42^{2(\nu-n)}\pi^{n+1}}2\int_{0}^{+\infty}d\lambda |C_{\nu}(\lambda)|^{-2}\int_{\mathbb{B}^{n}}(1-<z,w>)^{-\nu}(\phi^{(n-1,-\nu)}_{\lambda}(d(z,w)))F(w)d\mu_{\nu}(w)
$$

\begin{equation}\label{eq:comp4}
+\sum_{0}^{\frac{\nu-n}{2}}c_{j}\int_{\mathbb{B}^{n}}(1-<z,w>)^{-\nu}(\phi^{(n-1,-\nu)}_{\lambda_{j}}(d(z,w)))F(w)d\mu_{\nu}(w).
\end{equation}

Making use the change of variable $s=\lambda^{2}$ in the first integral of (\ref{eq:comp4}), and $s_{j}=\lambda_{j}^{2}$ in the discreet part then, (\ref{eq:comp4}) can be rewritten as:\\

$$
F(z)=\frac{\Gamma(n)}{4\pi^{n+1}2^{2(\nu-n)}}\int_{0}^{+\infty}ds |C_{\nu}(\sqrt{s})|^{-2}s^{\frac{-1}{2}}\int_{\mathbb{B}^{n}}(1-<z,w>)^{-\nu}(\phi^{(n-1,-\nu)}_{\sqrt{s}}(d(z,w)))F(w)d\mu_{\nu}(w)
$$

\begin{equation}\label{eq:comp5}
+\sum_{0}^{\frac{\nu-n}{2}}c_{j}\int_{\mathbb{B}^{n}}(1-<z,w>)^{-\nu}(\phi^{(n-1,-\nu)}_{\sqrt{s_{j}}}(d(z,w)))F(w)d\mu_{\nu}(w).
\end{equation}

This last identity can be written in the distributional sense as:
\begin{equation}\label{eq:comp6}
F(z)=\int_{-\infty}^{\infty}(\int_{\mathbb{B}^{n}}e(s,w,z)F(w)d\mu_{\nu}(w))ds,
\end{equation}
where the Schwartz kernel $e(s,w,z)$ is given by:
$$e(s,w,z)=\frac{\Gamma(n)}{4\pi^{n+1}2^{2(\nu-n)}}(1-<z,w>)^{-\nu}\chi_{+}(s)|C_{\nu}(\sqrt{s})|^{-2}(\sqrt{s})^{-1}\phi_{\sqrt{s}}^{(n-1,-\nu)}(d(z,w))$$

$$
+\sum_{j=0}^{\frac{\nu-n}{2}}c_{j}(1-<z,w>)^{-\nu}\phi_{\sqrt{s_{j}}}^{(n-1,-\nu)}(d(z,w)\delta(s-s_{j}).
$$
Now, by returning back to the equation (\ref{eq:comp1}) and applying the operator $\widetilde{\Delta_{\nu}}$ to its both sides, we obtain the following equation:
\begin{equation}\label{eq:comp11}
\widetilde{\Delta_{\nu}}[F](z)=\int^{+\infty}_{-\infty} \lambda^{2}\mathcal{P}_{\lambda}^{\nu}[F](z)d\lambda +
\sum_{0\leq j<\frac{ \nu-n}{2}}\lambda_{j}^{2}\mathcal{R}^{\nu}_{j}[F](z).
\end{equation}
As in the above, we make use of the change of variable $s=\lambda^{2},$ in the integral part of (\ref{eq:comp11}) and put $s_{j}=\lambda_{j}^{2}$ in the discreet part. Then by following the same steps, we obtain:
\begin{equation}\label{eq:comp21}
\widetilde{\Delta_{\nu}}[F](z)=\int_{-\infty}^{+\infty}s(\int_{\mathbb{B}}e(s,w,z)F(w)d\mu_{\nu}(w))ds.
\end{equation}
By considering the functional $T$ which corresponds to a test function $\varphi$ the operator $<T,\varphi> \in \mathbf{L}(\chi,\mathrm{L}_{\nu}^{2}(\mathbb{B}_{n}))$ defined by:
\begin{equation}
<T,\varphi>[F](z)=\int_{-\infty}^{+\infty}\varphi(s)(\int_{\mathbb{B}}e(s,w,z)F(w)d\mu_{\nu}(w))ds,
\end{equation}
we observe that the two equations (\ref{eq:comp6}) and (\ref{eq:comp21}) become:
\begin{equation}
<T,1>=I,
\end{equation}
\begin{equation}
<T,s>=\widetilde{\Delta_{\nu}},
\end{equation}
and by the uniqueness of the spectral density associated with a self-adjoint operator, we conclude that the functional $T$ is nothing but the spectral density of the operator $\widetilde{\Delta_{\nu}}.$ This end the proof.

\begin{remark}
For given a suitable function $f:\mathbb{R}\rightarrow\mathbb{C},$ the operator $f(\widetilde{\Delta_{\nu}}),$ is defined by

\begin{equation}\label{eq:functionofoperator}
f(\widetilde{\Delta_{\nu}})[\varphi](z)=\int_{\mathbb{B}^{n}}\Omega_{f}(w,z)\varphi(w)d\mu_{\nu}(w)
\end{equation}
where the kernel $\Omega_{f}(w,z)$ is defined by:

\begin{equation}\label{eq:kerneloffunctionofoperator}
\Omega_{f}(w,z)=\int_{\sigma(\widetilde{\Delta}_{\nu})}e(s,w,z)f(s)ds
\end{equation}

with (\ref{eq:functionofoperator}) and (\ref{eq:kerneloffunctionofoperator}) are understand in the distributional sense.
\end{remark}

As a direct consequence of the above proposition we can derive the heat kernel of the $G-$ invariant operator $\Delta_{\nu}$.Precisely, we have the following.
\begin{proposition}
Let $\psi(t,z),$ be the solution of the heat Cauchy problem associated to the operator $\Delta_{\nu},$ on $\mathbb{B}_{n}:$
\begin{equation}
\partial_{t}\psi(t,z)=\Delta_{\nu}\psi(t,z), (t,z)\in\mathbb{R}_{+}\times \mathbb{B}_{n}
\end{equation}
\begin{equation}
\psi(0,z)=\varphi(z)\in\mathcal{C}_{0}^{\infty}(\mathbb{B}_{n}).
\end{equation}
Then, $\psi(t,z),$ is given by the integral formula:
\begin{equation}
\psi(t,z)=\int_{\mathbb{B}_{n}}K_{\nu}(t,z,w)\varphi(w)d\mu_{\nu}(w),
\end{equation}

where$ K_{\nu}(t,z,w),$ is the heat kernel given by:
\begin{equation}
K_{\nu}(t,z,w)=(1-<z,w>)^{-\nu}\sum_{j=0}^{\frac{\nu-n}{2}}\tau_{j}e^{-2j(\nu-n-2j)t}P_{j}^{(n-1,-\nu)}(\cosh2d(z,w))
\end{equation}

\begin{equation}
+(1-<z,w>)^{-\nu}e^{-t(\nu-n)^{2}}\frac{\Gamma(n)}{2\pi^{n+1}2^{2(\nu-n)}}
\end{equation}
$$
\times \int_{0}^{+\infty}e^{-t\lambda^{2}}\mid C_{\nu}(\lambda)\mid^{-2}F(\frac{n-\nu-i\lambda}{2},\frac{n-\nu+i\lambda}{2},n;-\sinh^{2}(d(z,w)))d\lambda.
$$
where $\tau_{j}=\frac{2(\nu-n-2j)\Gamma(\nu-j)}{\pi^{n}\Gamma(\nu-n-j+1)}$ and $C_{\nu}(\lambda)$ is the Harish-chandra function defined in (\ref{eq:harichanda}).

\end{proposition}

\begin{proof}
The solution $\psi(t,z),$ is given by the action of the semigroup $e^{t\Delta_{\nu}}$ on the initial data $\varphi(z),$:

\begin{equation}
\psi(t,z)=e^{t\Delta_{\nu}}[\varphi](z).
\end{equation}
Note that for the operators $ \Delta_{\nu}$ and $\widetilde{\Delta_{\nu}}=-(\Delta_{\nu}+(n-\nu)^{2}),$ we have the following semigroup relation:

\begin{equation}
e^{t\Delta_{\nu}}=e^{-t(\nu-n)^{2}}e^{-t\widetilde{\Delta_{\nu}}}.
\end{equation}
Then, the heat kernel $k(t,z,w)$ of $\Delta_{\nu}$ is given by:
\begin{equation}
K(t,z,w)=e^{-t(\nu-n)^{2}}\widetilde{K}_{\nu}(t,z,w),
\end{equation}
where $\widetilde{K}_{\nu}(t,z,w)$ is the heat kernel of the operator $\widetilde{\Delta_{\nu}}.$

By using (\ref{eq:functionofoperator}) and (\ref{eq:kerneloffunctionofoperator}),we obtain:

$$
\widetilde{K}_{\nu}(t,z,w)=\int_{\sigma(\widetilde{\Delta_{\nu}})}e(s,w,z)e^{-ts}ds.
$$
Then, we have
$$
\widetilde{K}_{\nu}(t,z,w)=\sum_{j=0}^{\frac{\nu-n}{2}}c_{j}(1-<z,w>)^{-\nu}\phi_{\sqrt{s_{j}}}^{(n-1,-\nu)}(d(z,w))e^{-ts_{j}}
$$

\begin{equation}\label{eq:heatofdeltatilda}
+\frac{\Gamma(n)}{4\pi^{n+1}2^{(\nu-n)}}(1-<z,w>)^{-\nu}\int_{0}^{+\infty}|C_{\nu}(\sqrt{s})|^{-2}(\sqrt{s})^{-1}\phi_{\sqrt{s}}^{(n-1\nu)}(d(z,w))e^{-ts}ds.
\end{equation}
 Recall that $\sqrt{s_{i}}=\lambda_{j}=i(2j+n-\nu),$ $ j=0,1,...,\frac{n-\nu}{2}.$ Then the Jacobi function $\phi_{\sqrt{s_{j}}}^{(n-1,-\nu)}(d(z,w)),$ involved in the discreet part of the kernel $ \widetilde{K}_{\nu}(t,z,w),$ becomes

\begin{equation}\label{eq:jacobidensitiy1}
\phi_{\sqrt{s_{j}}}^{(n-1,-\nu)}(d(z,w))=_{2}F_{1}(-j,j+n-\nu,n;-\sinh^{2}(d(z,w)).
\end{equation}
Next,by using of the identity (\cite{magn:66}, p.39)
\begin{equation}
P_{k}^{(\alpha,\beta)}(y)=\frac{(1+\alpha)_{k}}{k!}_{2}F_{1}(-k,\alpha+\beta+k+1,\alpha+1;\frac{1-y}{2}),
\end{equation}
for $\alpha=n-1, \beta=-\nu, k=j$ and $y=1+2\sinh^{2}(d(z,w))=\cosh2d(z,w),$ the equation (\ref{eq:jacobidensitiy1}) becomes:

\begin{equation}
\phi_{\sqrt{s_{j}}}^{(n-1,-\nu)}(d(z,w))=\frac{j!}{(n)_{j}}P_{j}^{(\alpha,\beta)}(\cosh2d(z,w)).
\end{equation}

By inserting the above expression of Jacobi function in the discreet part of the equation (\ref{eq:heatofdeltatilda}) and using the change of variable $s=\lambda^{2}, \lambda>0,$ in the continuous part, we obtain:

\begin{equation}
\widetilde{K}_{\nu}(t,z,w)=(1-<z,w>)^{-\nu}\sum_{j=0}^{\frac{\nu-n}{2}}\tau_{j}e^{-(2j+n-\nu)^{2}t}P_{j}^{(n-1,-\nu)}(\cosh2d(z,w))
\end{equation}

\begin{equation}
+(1-<z,w>)^{-\nu}\frac{\Gamma(n)}{2\pi^{n+1}2^{2(\nu-n)}}
\end{equation}

$$
\times \int_{0}^{+\infty}e^{-t\lambda^{2}}\mid C_{\nu}(\lambda)\mid^{-2}F(\frac{n-\nu-i\lambda}{2},\frac{n-\nu+i\lambda}{2},n;-\sinh^{2}(d(z,w)))d\lambda,
$$

where $\tau_{j}=\frac{2(\nu-n-2j)\Gamma(\nu-j)}{\pi^{n}\Gamma(\nu-n-j+1)}.$ Then, we get the desired result.

\end{proof}

\section{The Berezin Transform}
  Let $l$ be a fixed integer in the set  $j=0,.....,[\frac{\nu-n}{2}],$ and let $A_{\nu}^{2,\nu}(\mathbb{B}^{n})=R^{\nu}_{l}L^{2}_{\nu}(\Bbb B^{n})$ be the subspace appearing in the discrete part of the Plancherel formula (\ref{eq:spectraldecomposition}). Then, according to \cite{zang:92}, $A_{\nu}^{2,\nu}(\mathbb{B}^{n})$ is a closed invariant subspace of $L^{2}_{\nu}(\Bbb B^{n}),$ with reproducing kernel given by:
\begin{equation}
K_{l}^{\nu}(z,w)=c_{l}\frac{1}{(1-<z,w>)^{\nu}}\quad_{2}F_{1}(-l,l-\nu+n,n;1-\frac{|1-<z,w>|^{2}}{(1-|z|^{2})(1-|w|^{2})}).
\end{equation}
where $c_{l}$ is the constant defined by (\ref{eq:constantreproduce}).
\begin{remark}
Thanks to the formula (\ref{eq:hyperbolicdistance}) The reproducing kernel $K_{l}^{\nu}(z,w),$ can be written also as:
\begin{equation}\label{eq:reproducedistance}
K^{\nu}_{l}(z,w)=c_{l}(1-<z,w>)^{-\nu}F(-l,l-\nu+n,n,-\sinh^{2}d(z,w)),
\end{equation}
\end{remark}

Notice that if  $l=0$ the above kernel is the Bergman kernel. Thus the space $A_{0}^{2,\nu}(\mathbb{B}^{n})$
is the classical weighted Bergman space of holomorphic functions in $L_{\nu}^{2}(\mathbb{B}^{n})$.\\
As is well known the classical Berezin transform associated to $A_{0}^{2,\nu}(\mathbb{B}^{n})$ is defined by:
$$
B_{\nu}F(z)=\frac{2\Gamma(\nu)}{\pi^{n}\Gamma(\nu-n)}\int_{\mathbb{B}^{n}}\frac{(1-|z|^{2})^{\nu}}{|1-<z,w>|^{2\nu}}F(w)d\mu_{\nu}(w),
$$
and has been studied by many authors.\\
As mentioned in the above, our aim in this section is to define a G- invariant Berezin transform associated to the G-invariant eigenspace $A_{l}^{2,\nu}(\mathbb{B}^{n}),$ to this end we consider the following kernel function:

\begin{equation}\label{eq:berkernel1}
B^{\nu}_{l}(z,w)=(1-<z,w>)^{-\nu}\frac{\mid K^{\nu}_{l}(z,w)\mid ^{2}}{K^{\nu}_{l}(z,z)K^{\nu}_{l}(w,w)}, (z,w)\in\mathbb{B}^{n}\times \mathbb{B}^{n}.
\end{equation}

\textbf{Definition}
Assume that $l\neq0.$ The transformation $B^{l}_{\nu}$ defined by:
$$L^{2}_{\nu}(\Bbb B^{n})\rightarrow L^{2}_{\nu}(\Bbb B^{n})$$

\begin{eqnarray}\label{eq:definition}
B^{l}_{\nu}F(z)=\int_{\mathbb{B}^{n}} F(w)B^{\nu}_{l}(z,w) d\mu_{\nu}(w),
\end{eqnarray}
is called here the Berezin transform on the the line bundle over the complex hyperbolic space  $\Bbb B^{n}=SU(n,1)/SU(U(n)\times U(1)).$ \\

Explicitly, from the expression of the reproducing kernel $K^{\nu}_{l}(z,w)$ given in (\ref{eq:reproducedistance}) combining with the use of the equation (\ref{eq:hyperbolicdistance}), the expression of the berezin kernel $B^{\nu}_{l}(z,w)$ is given by:
\begin{equation}\label{eq:berezinkernel}
B^{\nu}_{l}(z,w)=(1-<z,w>)^{-\nu}\cosh^{-2\nu} d(z,w)|F(-l,l-\nu+n,n,-\sinh^{2}d(z,w))|^{2}.
\end{equation}
It is easy to establish the following,

\begin{proposition}
The Berezin transform $B^{l}_{\nu}$ is $G$-invariant with respect to the representation $T^{\nu}$.
\end{proposition}
\textbf{Proof.} Let $$g^{-1}=\left(\begin{array}{cc}
                         a & b \\
                         c & d
                       \end{array}\right)\in G$$
\begin{equation}\label{eq:inv1}
T^{\nu}_{g}[B^{\nu}_{l}F](z)=(cz+d)^{-\nu}\int_{\Bbb B^{n}}(1-<g^{-1}.z,w>)^{-\nu}\frac{|K^{\nu}_{l}(g^{-1}.z,w)|^{2}}{K^{\nu}_{l}(g^{-1}.z,g^{-1}.z)K^{\nu}_{l}(w,w)}F(w)d\mu_{\nu}(w).
\end{equation}
By using the formula:
\begin{equation}\label{eq:inv2}
1-<g^{-1}z,g^{-1}\zeta>=\frac{(1-<z,\zeta>)}{(cz+d)(\overline{c\zeta+d})}
\end{equation}

for $z\in \Bbb B^{n},$ and $\zeta=g.w$, we obtain:
\begin{equation}\label{eq:inv3}
(1-<g^{-1}z,w>)^{-\nu}=(cz+d)^{\nu}(\overline{cg.w+d})^{\nu}(1-<z,g.w>)^{-\nu},
\end{equation}

\begin{equation}\label{eq:inv4}
K^{\nu}_{l}(g^{-1}.z,w)=(cz+d)^{\nu}(\overline{cg.w+d})^{\nu}K^{\nu}_{l}(z,g.w),
\end{equation}
and
\begin{equation}\label{eq:inv5}
K^{\nu}_{l}(g^{-1}.z,g^{-1}.z)=|cz+d|^{2\nu}K^{\nu}_{l}(z,z),
\end{equation}
Then, by inserting (\ref{eq:inv3}),(\ref{eq:inv4}) and (\ref{eq:inv5}) in (\ref{eq:inv1}) we obtain:

\begin{equation}\label{eq:inv6}
T^{\nu}_{g}[B^{\nu}_{l}F](z)=\int_{\Bbb B^{n}}(\overline{cg.w+d})^{\nu}|cg.w+d|^{2\nu}(1-<z,g.w>)^{-\nu}\frac{|K^{\nu}_{l}(z,g.w)|^{2}}{K^{\nu}_{l}(z,z)K^{\nu}_{l}(w,w)}F(w)d\mu_{\nu}(w).
\end{equation}
By using the change of variable $\zeta=g^{-1}w,$ The equation (\ref{eq:inv6}) becomes:
$$
T^{\nu}_{g}[B^{\nu}_{l}F](z)=\int_{\Bbb B^{n}}(\overline{c\zeta+d})^{\nu}|c\zeta+d|^{2\nu}(1-<z,\zeta>)^{-\nu}\frac{|K^{\nu}_{l}(z,\zeta)|^{2}}{K^{\nu}_{l}(z,z)K^{\nu}_{l}(g^{-1}\zeta,g^{-1}\zeta)}
$$

\begin{equation}\label{eq:inv7}
\times F(g^{-1}\zeta)d\mu_{\nu}(g^{-1}\zeta).
\end{equation}

After  inserting the equation (\ref{eq:inv5}) for which $z$ is replaced by $\zeta$ in the equation (\ref{eq:inv7}) and using the fact that the measure $d\mu_{\nu}(\zeta)=(1-|\zeta|^{2})^{\nu-n-1}dm(\zeta),$ is G-invariant, the equation (\ref{eq:inv7}) becomes

\begin{equation}
T^{\nu}_{g}[B^{\nu}_{l}F](z)=\int_{\Bbb B^{n}}(c\zeta+d)^{-\nu}(1-<z,\zeta>)^{-\nu}\frac{|K^{\nu}_{l}(z,\zeta)|^{2}}{K^{\nu}_{l}(z,z)K^{\nu}_{l}(\zeta,\zeta)} F(g^{-1}\zeta)d\mu_{\nu}(\zeta).
\end{equation}
$$=B^{\nu}_{l}[T^{\nu}_{g}F](z).$$ This ends the proof.\\

In order to prove that $B^{\nu}_{l},$ is a bounded operator, we have need the following lemma.
\begin{lemma}
The generalized berezin kernel $B^{\nu}_{l}(z,w)$ given in (\ref{eq:berezinkernel}) admits also the following expression:
\begin{equation}\label{eq:berezinth}
B^{\nu}_{l}(z,w)=(1-<z,w>)^{-\nu}\cosh^{(4l-2\nu)} (d(z,w))|F(-l,-l+\nu,n,\tanh^{2}(d(z,w)))|^{2}
\end{equation}
\end{lemma}
\textbf{Proof.}
From (\ref{eq:berezinkernel}), we have:
$$
B^{\nu}_{l}(z,w)=(1-<z,w>)^{-\nu}\cosh^{-2\nu} d(z,w)|F(-l,l-\nu+n,n,-\sinh^{2}d(z,w))|^{2}.
$$
Making use of the Euler formula (\cite{magn:66},p.47)
\begin{equation}\label{eq:Gaussformula}
F(a,b,c)=(1-z)^{-a}F(a,c-b,\frac{z}{z-1})
\end{equation}
for $a=-l, b=l+n-\nu$ and $c=n$, we obtain:

\begin{equation}
F(-l,l-\nu+n,n,-\sinh^{2}d(z,w))=ch^{2l}d(z,w)F(-l,\nu-l,n;\tanh^{2}d(z,w)),
\end{equation}
then by inserting (\ref{eq:Gaussformula}) in the above expression of $ B^{\nu}_{l}(z,w),$ we get the desired result. This ends the proof.\\

We have the following proposition:

\begin{proposition}. The Berezin transform $B^{l}_{\nu}$ is a bounded operator on $L^{p}(\mathbb{B}^{n},d\mu_{\nu})$ for $1\le p\le \infty$
\end{proposition}

Proof. We start by proving that $B^{l}_{\nu}$  is bounded operator in $L^{\infty}(\mathbb{B}^{n},d\mu_{\nu})$. \\

First observe that since $l$ is a positive integer then
$$
 |F(-l,-l+\nu,n,\tanh^{2}d(z,w))|^{2}\le M,
$$
for some positive constant $ M=M(\nu,l)$.
Thus
\begin{eqnarray}
\mid B^{l}_{\nu}(z,w)\mid \le M \mid 1-<z,w>\mid^{-3\nu} (1-\mid z\mid^{2})^{\nu}(1-\mid w\mid^{2})^{\nu} (\cosh^{2}d(z,w))^{2l},
\end{eqnarray}
from which we deduce that:
\begin{eqnarray}
\mid B^{l}_{\nu}(z,w)\mid \le M \frac{(1-\mid z\mid^{2})^{\nu-2l}(1-\mid w\mid^{2})^{\nu-2l}}{\mid 1-<z,w>\mid^{3\nu-4l}}.
\end{eqnarray}
Now we recall a result in \cite{ruwa:80} on the asymptotic behavior of certain integrals.

Let $c$ and $t$ be real numbers such $c>0$  and $t>-1$. Then,

\begin{eqnarray}\label{eq:convergenceintegral}
(1-\mid z\mid^{2})^{-c} \thickapprox
\int_{\mathbb{B}^{n}}\frac{(1-\mid w\mid^{2})^{t} dm(w)}{\mid 1-<z,w>\mid^{n+1+t+c}}.
\end{eqnarray}

In above the notation $a(z)\thickapprox b(z)$ means that the ratio $\frac{a(z)}{b(z)}$ has a positive limit as $\mid z\mid$ goes to $1$.\\

By using (\ref{eq:convergenceintegral}) with $t=2\nu-2l-n-1$ and $c=\nu-2l$ we deduce easily that,
\begin{eqnarray}\label{eq:integralestimate}
\int_{\mathbb{B}^{n}}\mid B^{l}_{\nu}(z,w)\mid d\mu_{\nu}(w)\le C,
\end{eqnarray}
for some positive constant $C=C_{\nu,l}$.\\
It follows from above that for every $F\in L^{\infty}(\mathbb{B}^{n},d\mu_{\nu})$ we have
$$
\parallel B^{l}_{\nu}F\parallel_{\infty}\le C  \parallel F\parallel_{\infty},
$$
therefore $B^{l}_{\nu}$ is bounded on $L^{\infty}(\mathbb{B}^{n},d\mu_{\nu})$.\\

Next let  $F\in L^{1}(\mathbb{B}^{n},d\mu_{\nu})$. Then,
$$
\parallel B^{l}_{\nu}F\parallel_{1}\le \int_{ \mathbb{B}^{n}}F(w)(\int_{ \mathbb{B}^{n}}\mid B^{l}_{\nu}(z,w)\mid d\mu_{\nu}(z))d\mu_{\nu}(w),
$$
it follows from (\ref{eq:integralestimate}) that
$$
\parallel B^{l}_{\nu}F\parallel_{1}\le  C \parallel F\parallel_{1},
$$
therefore the Berezin transform is bounded on $L^{1}(\mathbb{B}^{n},d\mu_{\nu})$ and the result follows from The Riesz-Thorin Theorem.

\section{The Berezin Transform as function of the invariant Laplacian.}

In this section, we shall express the transform $B^{l}_{\nu}$ as a function of the invariant Laplacian $\Delta_{\nu}$.
To this end we need the following lemma.\\
\begin{lemma}
 Let $j=0,1,.....,[\frac{\nu-n}{2}].$ Then,the complex numbers
\begin{equation}
  \xi_{j}=(\nu-n-2j)i, Im(\xi_{j})>0,
  \end{equation}
   are poles of the function $\eta(\lambda)=(C_{\nu}(\lambda)C_{\nu}(-\lambda))^{-1}$ and we have:
\begin{equation}\label{eq:linkformula}
c_{j}=2^{2(n-\nu)}\frac{\Gamma(n)}{\pi^{n}}(-iRes(\eta;\lambda=\xi_{j})),
\end{equation}
where $Res(\eta;\lambda=\xi_{j})$ means the residue of the function $\eta(\lambda)$ at $\lambda=\xi_{j}.$
and the constant $c_{j}$ is the constant defined by (\ref{eq:constantreproduce})
\end{lemma}
\textbf{Proof.} From the expression of $C_{\nu},$ the function $\eta(\lambda)$ can be written as
\begin{equation}
\eta(\lambda)=\frac{2^{2(\nu-n)}}{(\Gamma(n))^{2}}\frac{\Gamma(\frac{n+\nu+i\lambda}{2})\Gamma(\frac{n+\nu-i\lambda}{2})\Gamma(\frac{n-\nu-i\lambda}{2})}{\Gamma(i\lambda)\Gamma(-i\lambda)}\Gamma(\frac{n-\nu+i\lambda}{2}),
\end{equation}
is not difficult,to see that,
\begin{equation}
 Res(\eta;\lambda=\xi_{j})=\frac{2^{2(\nu-n)}}{(\Gamma(n))^{2}}\frac{\Gamma(\nu-j)\Gamma(j+n)\Gamma(j+n-\nu)}{\Gamma(\nu-n-2j)\Gamma(2j+n-\nu)}Res(\Gamma(\frac{n-\nu+i\lambda}{2});\lambda=\xi_{j})
\end{equation}
By a direct calculus we get,
$$Res(\Gamma(\frac{n-\nu+i\lambda}{2});\lambda=\xi_{j})=\frac{2i(-1)^{j}}{j!}.$$
Hence, we get:
\begin{equation}\label{eq:residue}
Res(\eta;\lambda=\xi_{j})=(-1)^{j}\frac{2^{2(\nu-n)+1}}{j!(\Gamma(n))^{2}}\frac{\Gamma(\nu-j)\Gamma(j+n)\Gamma(j+n-\nu)}{\Gamma(\nu-n-2j)\Gamma(2j+n-\nu)}i.
\end{equation}
Recall that the constant $c_{j},$ is given by
$$
c_{j}=\frac{2\Gamma(n+j)}{\pi^{n}\Gamma(n)j!}\frac{(\nu-n-2j)\Gamma(\nu-j)}{\Gamma(\nu-n-j+1)},
$$
From this expression of $c_{j}$ and the equation (\ref{eq:residue}) we obtain:
\begin{equation}\label{eq:quotient}
\frac{c_{j}}{Res(f;\lambda=\xi_{j})}=\frac{(-1)^{j}2^{2(n-\nu)}\Gamma(n)}{\pi^{n}}\frac{\Gamma(2j+n-\nu)\Gamma(1-(2j+n-\nu))}{\Gamma(j+n-\nu)\Gamma(1-(j+n-\nu))}(-i),
\end{equation}

By using the formula (\cite{magn:66},p.2),
$$\Gamma(z)\Gamma(1-z)=\frac{\pi}{sin \pi z},z\neq 0,-1,1,-2,2,......,$$
the equation (\ref{eq:quotient}) becomes,
\begin{equation}
\frac{c_{j}}{Res(f;\lambda=\xi_{j})}=\frac{(-1)^{j}2^{2(n-\nu)}\Gamma(n)}{\pi^{n}}\frac{\sin \pi(j+n-\nu)}{\sin \pi(2j+n-\nu)}(-i),
\end{equation}
$$=\frac{2^{2(n-\nu)}\Gamma(n)}{\pi^{n}}(-i).$$
Thus we get the desired formula (\ref{eq:linkformula}).This ends the proof.

Since the  the Berezin transform $B^{l}_{\nu}$ is a bounded operator on $L^{2}(\mathbb{B}^{n},d\mu_{\nu})$ commuting with the representation $T_{\nu}$ it follows that it is a function of the $G$-invariant Laplacian $\Delta_{\nu}$.\\
Namely there exists a $\mathbb{C}$-valued Borelian function $h$ on $\mathbb{R}$ such that
$$
 B^{\nu}_{l}=h(\Delta_{\nu}).
$$
 in other hand, since we dispose of the spectral function (\ref{eq:spectralfunction}) of the shifted $G-$ invariant Laplacian $\widetilde{\Delta_{\nu}}=-(\Delta_{\nu}+(n-\nu)^{2}),$ it will be natural to give $B^{\nu}_{l}$ in terms of $\widetilde{\Delta_{\nu}},$ instead of

 $\Delta_{\nu}$. That is

\begin{equation}
B^{\nu}_{l}=f(\widetilde{\Delta_{\nu}}),
\end{equation}

for some complex valued Borelian function $f$ on $\mathbb{R}$.\\

The main result of this paper is:
\begin{theorem}
Let $l=0,1,2,.....;[\frac{\nu-n}{2}],$ Then, the Berezin transform $B^{\nu}_{l}$ defined by (\ref{eq:berkernel1}) and (\ref{eq:definition})can be expressed in terms of the $G-$invariant operator $\Delta_{\nu}$ as:
$$
B^{\nu}_{l}=\frac{\pi^{n}\Gamma^{2}(n)}{2\Gamma^{2}(\nu-l)}|\Gamma(\frac{i\sqrt{-(\Delta_{\nu}+(n-\nu)^{2})}+3\nu-4l-n}{2})|^{2}
$$

\begin{equation}
\times\sum_{q=0}^{2l}(-1)^{q}A_{q}\frac{S_{q}(\frac{-(\Delta_{\nu}+(n-\nu)^{2})}{4},\frac{3\nu-4l-n}{2},\frac{\nu+n}{2},\frac{n-\nu}{2})}{\Gamma(2(\nu-l)+q)\Gamma(\nu-2l+q)}
\end{equation}

Where $S_{q}$ denotes the continuous dual Hahn polynomial.
\end{theorem}

\textbf{Proof.}
Recall from (\ref{eq:functionofoperator}) and (\ref{eq:kerneloffunctionofoperator}) that the operator $f(\widetilde{\Delta}_{\nu})$ is the Distributional operator with the Schwartz kernel:
\begin{equation}
\Psi_{g}(z,w)=\int_{\mathbb{R}}e(s,z,w)f(s)ds,
\end{equation}
where $e(s,z,w)$ is the spectral function (\ref{eq:spectralfunction}) associated to the operator $\widetilde{\Delta}_{\nu}$.

Then, the equation $B^{\nu}_{l}=f(\widetilde{\Delta}_{\nu}),$ implies that:
\begin{equation}
B^{\nu}_{l}(z,w)=\int_{\mathbb{R}} e(s,z,w)f(s)ds.
\end{equation}
Hence,
$$
B^{\nu}_{l}(z,w)=\frac{\Gamma(n)}{4\pi^{n+1}2^{2(\nu-n)}}(1-<z,w)^{-\nu}[\int_{0}^{+\infty}s^{\frac{-1}{2}}|C_{\nu}(S^{\frac{1}{2}}|^{-2}\phi^{n-1,-\nu}_{s^{\frac{1}{2}}}(t)f(s)ds
$$
\begin{equation}\label{eq:equalityofkernels}
+\sum_{l=0}^{\frac{\nu-n}{2}}c_{j}\phi_{s_{j}^{\frac{1}{2}}}^{(n-1,-\nu)}(t)f(s_{j})]
\end{equation}
where $s_{j}^{\frac{1}{2}}=\lambda_{j}$ and $\lambda_{j}$ are the parameters defined in (\ref{eq:lambdadeboussejra}).

Recall from (\ref{eq:berezinth}), that the Berezin kernel $B^{\nu}_{l}(z,w)$ has the following expression:
$$
B^{\nu}_{l}(z,w)=(1-<z,w>)^{-\nu}\cosh^{(4l-2\nu)} (d(z,w))|F(-l,-l+\nu,n,\tanh^{2}(d(z,w)))|^{2}.
$$
Then, after replacing the Berezin kernel $B^{\nu}_{l}(z,w)$ by its above expression in where we have set $t=d(z,w)$ , and using the change of variable $s=\lambda^{2},$  in the first integral of (\ref{eq:equalityofkernels}) and $s_{j}=\lambda_{j}^{2},$ in the discreet part, we obtain:

$$
\cosh^{(4l-2\nu)} (t)|F(-l,-l+\nu,n,\tanh^{2}(t))|^{2}=
$$

$$
\frac{\Gamma(n)}{4\pi^{n+1}2^{2(\nu-n)}}
\int_{0}^{+\infty}2|C_{\nu}(\lambda)|^{-2}\phi^{n-1,-\nu}_{\lambda}(t)f(\lambda^{2})d\lambda
$$

\begin{equation}
+\sum_{j=0}^{\frac{\nu-n}{2}}c_{j}\phi_{\lambda_{j}}^{(n-1,-\nu)}(t)f(\lambda_{j}^{2}).
\end{equation}

Putting $ g(\lambda)=f(\lambda^{2}).$ Then, the above equation becomes:

$$
\frac{\Gamma(n)}{2\pi^{n+1}2^{2(\nu-n)}}\int_{0}^{+\infty}\mid C_{\nu}(\lambda)\mid^{-2}\phi^{(n-1,-\nu)}_{\lambda}(t)g(\lambda)d\lambda
+\sum_{j=0}^{\frac{\nu-n}{2}}c_{j}\phi_{\lambda_{j}}^{(n-1,-\nu)}(t)g(\lambda_{j}).
$$

\begin{equation}\label{eq:equationintegral1}
=\cosh^{(4l-2\nu)} (t)|F(-l,-l+\nu,n,\tanh^{2}(t))|^{2}.
\end{equation}

Now, by using the fact :$\phi^{(n-1,-\nu)}_{\lambda}(t)=\phi^{(n-1,-\nu)}_{-\lambda}(t)$ and $g(\lambda)=g(-\lambda)$ and replacing the constants $c_{j}$ by theirs expressions given in the equation (\ref{eq:linkformula}) with taking into a count that $\zeta_{j}=-\lambda_{j}$ the equation (\ref{eq:equationintegral1}) becomes:

$$
\frac{1}{2\pi}\int_{0}^{+\infty}\mid C_{\nu}(\lambda)\mid^{-2}\phi^{(n-1,-\nu)}_{\lambda}(t)g(\lambda)d\lambda
$$

$$
+\sum_{\zeta_{j}\in D_{\nu}} (-iRes\left((C_{\nu}(\lambda)C_{\nu}(-\lambda))^{-1}),\lambda=\xi_{j}\right)\phi_{\zeta_{j}}^{(n-1,-\nu)}(t)g(\zeta_{j}),
$$

\begin{equation}\label{eq:integralequation}
=\frac{2^{2(\nu-n)}\pi^{n}}{\Gamma(n)}(\cosh t)^{-\mu}(F(-l,\nu-l,n,\tanh^{2}t))^{2},
\end{equation}

where $\mu=2\nu-4l,$ and the set $D_{\nu}$ is given by,
$$
 D_{\nu}=\{\zeta_{j}=i(\nu-n-2j),j=0,1,2,...., \nu-n-2j>0\}.
 $$

Now, we recall from \cite{koor:84},\cite{kaji:01}, some properties of the Fourier-Jacobi transform.
Assume that $\alpha>1$ and $|\beta|>\alpha+1.$ Then, the Fourier-Jacobi transform of a $C^{\infty}-$ compactly supported function $f$ on $\mathbb{R}$ is defined by

\begin{equation}\label{eq:jacobitransform}
\widehat{\varphi}(\lambda)=\int_{0}^{+\infty}\phi(t)\phi_{\lambda}^{(\alpha,\beta)}(t)\Delta_{\alpha,\beta}(t)dt
\end{equation}
its inverse is given by:
\begin{equation}\label{eq:inversejacobietransform}
\varphi(t)=\frac{1}{2\pi}\int_{0}^{+\infty}\widehat{\varphi}(\lambda)|C_{\alpha,\beta}(\lambda)|^{-2}d\lambda+\sum_{\lambda\in D\alpha,\beta}d_{\alpha,\beta}(\lambda)\widehat{\varphi}(\lambda)
\end{equation}
Where,
\begin{equation}
d_{\alpha,\beta}(\lambda)=-iRes(C_{\alpha,\beta}(z)C_{\alpha,\beta}(z)^{-1},z=\lambda)\phi_{\lambda}^{(\alpha,\beta)}(\lambda),
\end{equation}
with
\begin{equation}
D_{\alpha,\beta}=\{i(|\beta|-\alpha-1-2j):j=0,...,|\beta|-\alpha-1-2j>0\}
\end{equation}

\begin{equation}
\Delta_{\alpha,\beta}(t)=(2\sinh t)^{2\alpha+1}(2\cosh t)^{2\beta+1},
\end{equation}

\begin{equation}\label{eq:harichandraalphabeta}
C_{\alpha,\beta}(\lambda)=2^{\rho-i\lambda}\frac{\Gamma(\alpha+1)\Gamma(i\lambda)}{\Gamma(\frac{\alpha+\beta+1+i\lambda}{2})\Gamma(\frac{\alpha-\beta+1+i\lambda}{2})},\rho=\alpha+\beta+1
\end{equation}

and $\phi_{\lambda}^{(\alpha,\beta)}(t),$ is the Jacobi function defined in the proposition (\ref{eq:spectralfunction}).\\

Here, from (\ref{eq:inversejacobietransform}), is not difficult to see that the the left hand side of the integral equation, (\ref{eq:integralequation}),
is nothing other than the inverse of Fourier-Jacobi transform of the function $g,$ with $\alpha=n-1;\beta=-\nu.$

In other hand, the function:
\begin{equation}
h(t)=\frac{2^{2(\nu-n)}\pi^{n}}{\Gamma(n)}(\cosh t)^{-\mu}(F(-l,\nu-l,n,\tanh^{2}t))^{2}
\end{equation}
given in the right hand side of the equation (\ref{eq:integralequation}) can be written as:

\begin{equation}
 h(t)=(\cosh t)^{-\mu}\psi(\cosh^{-2}t),
\end{equation}

where $\mu=2\nu-4l$ and $\psi(t),$ is the $C^{\infty}$ function on $[0,1]$ given by:

\begin{equation}
\psi(t)=\frac{2^{2(\nu-n)}\pi^{n}}{\Gamma(n)}(F(-l,\nu-l,n,1-t))^{2}.
\end{equation}

More, in our case $\alpha=n-1,\beta=-\nu,$ it is easy to see that we have the inequality $\mu>\rho,$ where $\rho,$ is the parameter defined in (\ref{eq:harichandraalphabeta}). Then, thanks to (\cite{kaji:01}, $ \widehat{h}(\lambda)$ is well defined on $\mathbb{R}~$. Then by (\ref{eq:jacobitransform}) the function $g$ involved in the integral equation (\ref{eq:integralequation}) is given by:

$$
g(\lambda)=\int_{0}^{+\infty}h(t)\phi_{\lambda}^{(n-1,-\nu)}(t)\Delta_{n-1,-\nu}(t)dt
$$

$$
=\frac{\pi^{n}}{\Gamma(n)}
\int_{0}^{+\infty}\cosh ^{-\mu}t\mid F(-l,\nu-l,n,\tanh^{2} t)\mid^{2}\phi^{(n-1,-\nu)}_{\lambda}(t)
$$

\begin{equation}\label{eq:expression1ofg}
\times(\sinh t)^{2n-1}(\cosh t)^{1-2\nu}dt.
\end{equation}

Using the expression of Jacobi-polynomials (\cite{magn:66}, p.39)

\begin{equation}
P^{(\alpha,\beta)}_{k}(x)=\frac{(\alpha+1)_{k}}{k!}F(-k,\alpha+\beta+1+k,\alpha+1;\frac{1-x}{2}),
\end{equation}

for $k=l,$ $\alpha=n-1,$ $\beta=\nu-n-2l$ and $x=1-2\tanh t,$ we obtain:

\begin{equation}\label{eq:hypjacobi}
F(-l,\nu-l,n,\tanh ^{2}t)= \frac{l!}{(n)_{l}}P^{(n-1,\nu-n-2l)}_{l}(1-2\tanh ^{2}t).
\end{equation}

Then, by using (\ref{eq:hypjacobi}), the Equation (\ref{eq:expression1ofg}) becomes:

$$
g(\lambda)=C_{n,l,\nu}\int_{0}^{+\infty}(\cosh t)^{(4l-2\nu)}( P^{(n-1,\nu-n-2l)}_{l}(1-2\tanh^{2} t))^{2}
$$

\begin{equation}
\times \phi^{(n-1,-\nu)}_{\lambda}(t)(\sinh t)^{2n-1}(\cosh t)^{1-2\nu}dt.
\end{equation}
where,

\begin{equation}
C_{n,l,\nu}=\frac{\pi^{n}(l!)^{2}}{(n)_{l}^{2}\Gamma(n)}
\end{equation}

Use again (\ref{eq:Gaussformula}) to rewrite the Jacobi function $\phi_{\lambda}^{(n-1,-\nu)}$ as follows

$$
\phi_{\lambda}^{(n-1,-\nu)}(t)=(1-\tanh ^{2}t)^{\frac{i\lambda+n-\nu}{2}}F(\frac{i\lambda+n-\nu}{2},\frac{i\lambda+n+\nu}{2},n,\tanh ^{2}t).
$$

Henceforth

$$
g(\lambda)=C_{n,l,\nu}\int_{0}^{+\infty}(1-\tanh ^{2}t)^{\frac{i\lambda-n-4l+3\nu}{2}}\{P^{(n-1,\nu-n-2l)}_{l}(1-2\tanh ^{2}t)\}^{2}
$$

\begin{equation}
\times F(\frac{i\lambda+n-\nu}{2},\frac{i\lambda+n+\nu}{2},n,\tanh ^{2}t)\tanh ^{2n-1}tdt.
\end{equation}

Next make the change of variable $y=\tanh ^{2}t$ to rewrite the above integral as,
$$
g(\lambda)=\frac{C_{n,l,\nu}}{2}\int_{0}^{1}(1-y)^{\frac{i\lambda-n-4l+3\nu-2}{2}}y^{n-1}\{(P^{(n-1,\nu-n-2l)}_{l}(1-2y)\}^{2}
$$

\begin{equation}
\times F(\frac{i\lambda+n-\nu}{2},\frac{i\lambda+n+\nu}{2},n,y)dy.
\end{equation}

By using the following formula power expansion for the product of two Jacobi polynomials \cite{tapa:87} we get:

\begin{equation}
\{(P^{(n-1,\nu-n-2l)}_{l}(1-2y)\}^{2}=\frac{\Gamma^{2}(l+n)}{l!^{2}\Gamma^{2}(\nu-l)}\sum^{2l}_{q=0}(-1)^{q}A_{q}(2y)^{q},
\end{equation}

where

\begin{equation}
A_{q}=2^{-q}\sum_{p=\max (0,q-l)}^{\min (l,q)}(_{q-p}^{l})(_{p}^{l})\frac{\Gamma(\nu-l+q-p)\Gamma(\nu-l+p)}{\Gamma(n+p)\Gamma(n+q-p)}.
\end{equation}

Hence, we are lead to compute the following integral:

\begin{equation}
I_{q}(\lambda)=\int_{0}^{1}(1-y)^{\frac{i\lambda-n-4l+3\nu-2}{2}}y^{n+q-1} F(\frac{i\lambda+n-\nu}{2},\frac{i\lambda+n+\nu}{2},n,y)dy.
\end{equation}

For this, use the following  identity (\cite{grry:07} p:813)

\begin{equation}
\int^{1}_{0}x^{\rho-1}(1-x)^{\sigma-1}F(\alpha,\beta,\gamma;x)dx=\frac{\Gamma(\rho)\Gamma(\sigma)}{\Gamma(\rho+\sigma)}\\
_{3}F_{2}(\alpha,\beta,\rho;\gamma,\rho+\sigma;1),
\end{equation}

for all $\alpha,\beta,\gamma,\rho$ and $\sigma$ such that $\Re\rho>0, \Re\sigma>0$ and $\Re(\gamma+\sigma-\alpha-\beta)>0$, to get:

$$
I_{q}(\lambda)=\frac{\Gamma(n+q)\Gamma(\frac{i\lambda+3\nu-4l-n}{2})}{\Gamma(\frac{i\lambda+3\nu-4l+n}{2}+q)}
$$

\begin{equation}\label{eq:functionI1}
\times _{3}F_{2}(\frac{i\lambda+n-\nu}{2},\frac{i\lambda+n+\nu}{2},n+q;n,\frac{i\lambda+3\nu-4l+n}{2}+q;1).
\end{equation}
\\
Thus, we have established that:
\begin{equation}
g(\lambda)=\frac{\Gamma(n)\pi^{n}}{2\Gamma^{2}(\nu-l)}\sum_{0}^{2l}(-1)^{q}A_{q}I_{q}(\lambda).
\end{equation}
We will use the following identity on hypergeometric functions (\cite{bryu:08},p593)
$$
_{3}F_{2}(a,b,c;d,e;1)=\dfrac{\Gamma(d)\Gamma(e)\Gamma(d+e-c-a-b)}{\Gamma(c)\Gamma(d+e-a-c)\Gamma(d+e-b-c)}
$$
\begin{equation}
\times_{3}F_{2}(d-c,e-c,d+e-c-a-b;d+e-a-c,d+e-b-c;1);
\end{equation}

to rewrite the function $I_{q}(\lambda)$ given in equation (\ref{eq:functionI1}) as:

$$
I_{q}(\lambda)=\frac{\Gamma(n)\Gamma(\frac{-i\lambda+3\nu-4l-n}{2})\Gamma(\frac{i\lambda+3\nu-4l-n}{2})}{\Gamma(2(\nu-l))\Gamma(\nu-2l)}
$$

\begin{equation}\label{eq:functionI2}
\times _{3}F_{2}(-q,\frac{i\lambda+3\nu-4l-n}{2},\frac{-i\lambda+3\nu-4l-n}{2};2(\nu-l),\nu-2l;1).
\end{equation}

By using the formula (\cite{RKRS:98}, p.29)

\begin{equation}\label{eq:Hahnpolynomial}
\quad_{3}F_{2}(-j,a+ix,a-ix;a+b,a+c;1)=\frac{S_{j}(x^{2},a,b,c)}{(a+b)_{j}(a+c)_{j}}
\end{equation}

for $x=\frac{\lambda}{2},$ $a=\frac{3\nu-4l-n}{2},$ $b=\frac{\nu+n}{2},$ $c=\frac{n-\nu}{2}$ and $j=q,$

The above hypergeometric involved in (\ref{eq:functionI2}) becomes:
\begin{equation}
_{3}F_{2}(-q,\frac{i\lambda+3\nu-4l-n}{2},\frac{-i\lambda+3\nu-4l-n}{2}-;2(\nu-l),\nu-2l;1)=\frac{S_{q}(\frac{\lambda^{2}}{4},\frac{3\nu-4l-n}{2},\frac{n-\nu}{2})}{(2(\nu-l))_{q}((\nu-2l)_{q}}
\end{equation}

where $S_{j}(x^{2},a,b,c),$ is the continuous dual Hahn polynomial.

Henceforth, we obtain:
\begin{equation}
I_{q}(\lambda)=\frac{\Gamma(n)\Gamma(\frac{i\lambda+3\nu-4l-n}{2})\Gamma(\frac{-i\lambda+3\nu-4l-n}{2})}{\Gamma(2(\nu-l)+q)\Gamma(\nu-2l+q)}S_{q}(\frac{\lambda^{2}}{4},\frac{3\nu-4l-n}{2},\frac{\nu+n}{2},\frac{n-\nu}{2})
\end{equation}

finally, we get,
\begin{equation}\label{eq:gfinal}
g(\lambda)=\frac{\pi^{n}\Gamma^{2}(n)}{2\Gamma^{2}(\nu-l)}|\Gamma(\frac{i\lambda+3\nu-4l-n}{2})|^{2}\sum_{q=0}^{2l}(-1)^{q}A_{q}\frac{S_{q}(\frac{\lambda^{2}}{4},\frac{3\nu-4l-n}{2},\frac{\nu+n}{2},\frac{n-\nu}{2})}{\Gamma(2(\nu-l)+q)\Gamma(\nu-2l+q)}.
\end{equation}

Using the the fact $f(s)=g(s^{\frac{1}{2}})$, with $s=\lambda^{2},$ Hence, if we replace the spectral parameter $s$ by $-(\Delta_{\nu}+(n-\nu)^{2})$, we obtain the desired result. This ends the proof.

\section{The Berezin heat kernel.}
As a direct consequence of the above theorem, we can derive easily the Heat kernel of the Berezin operator $B^{\nu}_{l}.$ Precisely we have the following.
\begin{proposition}
Let $u(t,z),$ be the solution of the heat Cauchy problem associated to the $B_{l}^{\nu},$ on $\mathbb{B}^{n},$

\begin{equation}
(\partial_{t}+B_{l}^{\nu})u(t,z)=0, (t,z)\in \mathbb{R}\times\mathbb{B}^{n}
\end{equation}

\begin{equation}
u(0,z)=\varphi(z), \varphi\in C_{0}^{\infty}(\mathbb{B}^{n}).
\end{equation}
Then, $u(t,z),$ is given by the integral formula,
\begin{equation}
u(t,z)=\int_{\mathbb{B}^{n}}H^{\nu}_{l}(t,z,w)\varphi(w)d\mu_{\nu}(w),
\end{equation}
where $H^{\nu}_{l}(t,z,w),$ is the heat kernel given by:
$$
H^{\nu}_{l}(t,z,w)=(1-<z,w>)^{-\nu}\{\frac{\Gamma(n)}{2\pi^{n+1}2^{2(\nu-n)}}
\int_{0}^{+\infty}|C_{\nu}(\lambda)|^{-2}\phi^{n-1,-\nu}_{\lambda}(d(z,w))\exp(tg(\lambda)d\lambda\}
$$

\begin{equation}
+(1-<z,w>)^{-\nu}\{\sum_{j=0}^{\frac{\nu-n}{2}}c_{j}\phi_{\lambda_{j}}^{(n-1,-\nu)}(d(z,w))\exp(tg(\lambda_{j}))\}.
\end{equation}
where
\begin{equation}
g(\lambda)=\frac{\pi^{n}\Gamma^{2}(n)}{2\Gamma^{2}(\nu-l)}|\Gamma(\frac{i\lambda+3\nu-4l-n}{2})|^{2}\sum_{q=0}^{2l}(-1)^{q}A_{q}\frac{S_{q}(\frac{\lambda^{2}}{4},\frac{3\nu-4l-n}{2},\frac{\nu+n}{2},\frac{n-\nu}{2})}{\Gamma(2(\nu-l)+q)\Gamma(\nu-2l+q)},
\end{equation}

Where $S_{j}$ is the continuous  dual Hahn polynomial defined by (\ref{eq:Hahnpolynomial}), $c_{j}$ the constant defined by (\ref{eq:constantreproduce}) and $\lambda_{j}=i(2j+n-\nu).$
\end{proposition}

\textbf{Proof.}

It is not difficult to see that the solution $u(t,z),$ is given by the action of the semigroup $e^{-tB^{\nu}_{l}}$ on the initial data $\varphi(z),$,
\begin{equation}
u(t,z)=e^{-tB^{\nu}_{l}}[\varphi](z)=h_{t}(\Delta_{\nu}),h_{t}(s)=\exp(-tf(s)),
\end{equation}
where $f(s)$ is the function given by $f(s)=g(s^{\frac{1}{2}})$ and $g$ the function given in (\ref{eq:gfinal}). Then, by using (\ref{eq:functionofoperator}) and (\ref{eq:kerneloffunctionofoperator}), a direct calculus gives the above expression of the heat kernel.




\end{document}